\documentclass[reqno]{amsart}
\usepackage[ansinew]{inputenc}
\usepackage[T1]{fontenc}
\usepackage{amsmath}
\usepackage{amsfonts}
\usepackage{amssymb}
\usepackage{amsthm}
\usepackage{enumitem}

\def\N{{\mathbb N}}

\def\R{{\mathbb R}}

\def\i{\infty}
\def\e{\varepsilon}
\def\om{\omega}
\def\p{\partial}

\def\st{\, \middle| \,}
\def\weaklyto{\rightharpoonup}

\def\BB{{\mathcal B}}

\DeclareMathOperator*{\Argmin}{Argmin}

\DeclareMathOperator*{\ch}{co}

\DeclareMathOperator*{\dom}{dom}

\DeclareMathOperator*{\interior}{int}
\DeclareMathOperator*{\Prob}{Pr}

\newtheorem{corollary}{Corollary}

\newtheorem{lemma}{Lemma}
\newtheorem{proposition}{Proposition}
\newtheorem{theorem}{Theorem}

\def\XXint#1#2#3{{\setbox0=\hbox{$#1{#2#3}{\int}$} 
		\vcenter{\hbox{$#2#3$}}\kern-.5\wd0}}

\title{Unique Minimizers and the Representation of Convex Envelopes in Locally Convex Vector Spaces}

\author{Thomas Ruf and Bernd Schmidt}

\address[Thomas Ruf]{Institut f\"ur Mathematik, Universit\"at Augsburg, 86135 Augsburg, Germany}
\email{thomas.ruf@math.uni-augsburg.de}

\address[Bernd Schmidt]{Institut f\"ur Mathematik, Universit\"at Augsburg, 86135 Augsburg, Germany}
\email{bernd.schmidt@math.uni-augsburg.de}

\begin{document}

\maketitle

\begin{abstract}

It is well known that a strictly convex minimand admits at most one minimizer. We prove a partial converse: Let $X$ be a locally convex Hausdorff space and $f \colon X \mapsto \left( - \i , \i \right]$ a function with compact sublevel sets and exhibiting some mildly superlinear growth. Then each tilted minimization problem
\begin{equation} \label{eq. minimization problem}
	\min_{x \in X} f(x) - \langle x' , x \rangle_X
\end{equation}
admits at most one minimizer as $x'$ ranges over $\dom \left( \p f^* \right)$ if and only if the biconjugate $f^{**}$ is essentially strictly convex and agrees with $f$ at all points where $f^{**}$ is subdifferentiable. We prove this via a representation formula for $f^{**}$ that might be of independent interest.

\end{abstract}

\section{Introduction}

The minimizer of a strictly convex function $f$ is unique since for any minimizers $x_0 \not = x_1$ holds
$$
f \left( \lambda x_1 + \left( 1 - \lambda \right) x_0 \right) < \lambda f \left( x_1 \right) + \left( 1 - \lambda \right) f \left( x_0 \right) = \inf f \quad \forall \, \lambda \in \left( 0, 1 \right).
$$
Using subdifferential calculus, a slightly refined uniqueness criterion may be derived requiring merely that $f$ be essentially strictly convex, i.e. proper, convex and strictly convex on each line segment contained in $\dom \left( \p f \right)$. The simplicity of these considerations tempts to conjecture that more elaborate general uniqueness criteria for minimizers might exist. As far as the inhomogeneous problem (\ref{eq. minimization problem}) is concerned, this turns out to be wrong in the following precise sense: In order to have a decent existence theory for (\ref{eq. minimization problem}), it seems reasonable to require that linear perturbations of $f \colon X \mapsto \left( - \i , \i \right]$ have compact sublevel sets. Under this condition, we shall prove that the tilted minimization problem (\ref{eq. minimization problem}) admits at most one minimizer for each $x' \in \dom \left( \p f^* \right)$ if and only if $f$ agrees with its biconjugate $f^{**}$ on $\dom \left( f^{**} \right)$, which is then essentially strictly convex. This implies $f = f^{**}$ if $X$ is a Banach space. Therefore essential strict convexity is sufficient and necessary for uniqueness of minimizers in (\ref{eq. minimization problem}). An interesting consequence is that a possible failure of uniqueness in the pertaining inhomogeneous inclusion
\begin{equation} \label{eq. gradient system}
	x' \in \p f \left( x \right)
\end{equation}
cannot be restored by employing global minimality in (\ref{eq. minimization problem}) as a selection criterion.

The essential auxiliary tool in our proof will be a representation formula for the biconjugate $f^{**}$ that we prove beforehand. As is well known, there already exist formulas relating $f^{**} \left( x \right)$ for $x \in X$ with $f$ via
$$
f^{**} \left( x \right) = \liminf_{y \to x} \inf \left\{ \sum_{k = 1}^{N} \lambda_k f \left( y_k \right) \st N \in \N, \sum_{k = 1}^{N} \lambda_k = 1, \lambda_k \ge 0, \sum_{k = 1}^{N} \lambda_k y_k = y \right\}.
$$
Our new contribution consists of identifying general sufficient conditions under which the limit may be omitted and the infimum attains.
This question has already been investigated for $X = \R^d$, where Carathéodory’s Theorem bounds the number of points that contribute meaningfully to a convex combination. Obviously, this no longer works if $\dim X = \i$. We solve this problem by permitting probability measures as continuous convex combinations. The representation result thus obtained will allow simple rigorous proofs of several intuitive relationships between $f$ and $f^{**}$, from which our main result will eventually follow. We consider it likely that the representation formula has applications beyond the present setting and therefore might be of independent interest.

We do not know a result resembling our main theorem except \cite[Thm. 1]{Adequate Functions}, where a related result is proved in the particular case  of a reflexive Banach space in its weak topology. Our method of proof differs strongly. After our main result will have been proved, we shall obtain \cite[Thm. 1]{Adequate Functions} as a corollary.
\\

\textit{Remark on notation:}
Throughout, $T$ is a topological space with Borel $\sigma$-algebra $\BB(T)$ and $\Prob(T)$ are the Borel probability measures on $T$. 
For $t \in T$, let $\delta_t$ be the Dirac-measure supported at $t$. Let $V$ be a (topological) vector space. We denote by $V'$ its topological dual space. If $M \subset V$, then $\ch M$ and $\overline{\ch} M$ are the convex hull and closed convex hull of $M$. Similarly, for a function $f \colon V \mapsto \left[ - \i , \i \right]$, $\ch f$ and $\overline{\ch} f$ are the (closed) convex envelope of $f$, i.e. the largest (lower semi-continuous) convex function below $f$. For subsets $A \subset B$ of a fixed superset $B$, we write $A^c = B \setminus A$. The symbol $\chi_A$ denotes the function with value $1$ on $A$ and $0$ elsewhere, $I_A$ is the function with value $0$ on $A$ and $\i$ on $A^c$. We denote by $\p f$ the Fenchel-Moreau subdifferential of convex analysis. Moreover
$$
\dom \left( f \right) = \left\{ v \in V \st f \left( v \right) \in \R \right\} \text{ and }
\dom \left( \p f \right) = \left\{ v \in V \st \p f \left( v \right) \not = \emptyset \right\}.
$$
We say that $f$ is essentially strictly convex iff $f$ is proper, convex everywhere and strictly convex on each convex subset of $\dom \left( \p f \right)$. This is the notion of essential strict convexity introduced in \cite{Convex Analysis}. We caution the reader that we never mean essential strict convexity in the sense of \cite{Essential in B-Spaces} unless explicitly stated.

\section{Biconjugate Representation}

In this section, we shall obtain the announced representation formula for the biconjugate. Its proof will require a lower semi-continuity result for integral functionals with respect to the weak convergence of measures, to be applied when $T$ is a non-metrizable locally convex Hausdorff space. As this result is usually only proved for when $T$ is a metric space (via Lipschitz regularization of the integrand), we first provide here a proof that works for every topological space. We refer the reader to \cite[Ch. 4]{Conv Measures} for basic definitions.

\begin{proposition} \label{prop. lsc. integrand induces lsc. functional}
	Let $\psi \colon T \mapsto \left( - \i , \i \right]$ be lower semi-continuous and bounded below. Consider the integral functional
	$$
	I_\psi \colon \Prob(T) \mapsto \left( - \i , \i \right] \colon \mu \mapsto \int \psi \, d \mu.
	$$
	If a net $\mu_\alpha \in \Prob(T)$ converges weakly to a $\mu \in \Prob(T)$ that is $\tau$-additive (e.g. is Radon), then
	$$
	\liminf_\alpha I_\psi \left( \mu_\alpha \right) \ge I_\psi \left( \mu \right).
	$$
\end{proposition}

\begin{proof}
	We may assume $\psi \ge 0$. We will obtain our claim by expressing $I_\psi$ as a supremum of functionals fulfilling the same lower semi-continuity property. For $m \in \N$, consider the approximant
	$$
	\psi_m^N(x) := \sum_{n = 1}^N 2^{-m} \chi_{ \left\{ \psi > 2^{-m} n \right\} }(x).
	$$
	The sets $\left\{ \psi > 2^{-m} n \right\}$ being open by lower semi-continuity, the function $\psi_m^N$ is bounded, lower semi-continuous and hence induces an integral functional for which the desired property holds by \cite[Cor. 4.3.4]{Conv Measures}. We set
	$$
	\psi_m := \sup_{N \in \N} \psi_m^N = \lim_{N \to \i} \psi_m^N.
	$$
	The sequence $\psi_m$ increases since
	\begin{align*}
		\psi_m
		= \sum_{n = 1}^\i 2^{-m} \chi_{ \left\{ \psi > 2^{-m} n \right\} }
		& \le \sum_{n = 1}^\i 2^{-m} \left( \frac{\chi_{ \left\{ \psi > 2^{- m - 1} \left( 2 n - 1 \right) \right\} } + \chi_{ \left\{ \psi > 2^{- m - 1} 2 n \right\} }}{2} \right) \\
		& = \sum_{n = 1}^\i 2^{- m - 1} \chi_{ \left\{ \psi > 2^{- m - 1} n \right\} } \\
		& = \psi_{m + 1}.
	\end{align*}
	Moreover, if $2^{-m} n < \psi(x) \le 2^{-m} \left( n + 1 \right)$, then $\psi_m(x) = 2^{-m} n$, so that $\psi_m \uparrow \psi$ as $m \uparrow \i$. Therefore monotone convergence implies
	$$
	\sup_{N \in \N} \sup_{m > 0} \int \psi_m^N \, d \mu = \sup_{m > 0} \int \psi_m \, d \mu = \int \psi \, d \mu
	$$ 
	for every $\mu \in \Prob(T)$.
\end{proof}

We are now prepared to prove

\begin{lemma} \label{lem. biconjugate representation}
	Let $X$ be a locally convex Hausdorff space and $f \colon X \mapsto \left( - \i , \i \right]$ a function with (closed) compact sublevel sets. For $x \in X$ holds the representation
	\begin{equation} \label{eq. inf repr of co}
		f^{**} \left( x \right) = \liminf_{y \to x} \inf \left\{ \int f \, d \mu \st \mu \in \Prob(X), \int \om \, d \mu \left( \om \right) = y \right\}.
	\end{equation}
	The expectation in (\ref{eq. inf repr of co}) is to be understood as a Pettis integral, cf. \cite[Def. 3.26]{Ana}. If $\dom \left( f^* \right) = X'$ or if $f^{**} \left( x \right) = \inf f$ and $\dom \left( f^* \right)$ contains a balanced absorbent subset, then (\ref{eq. inf repr of co}) simplifies to
	\begin{equation} \label{eq. min repr of co}
		f^{**} \left( x \right) = \min \left\{ \int f \, d \mu \st \mu \in \Prob(X), \int \om \, d \mu \left( \om \right) = x \right\}.
	\end{equation}
	At least one minimizer in (\ref{eq. min repr of co}) is then a Radon measure. If $\mu_x$ minimizes in (\ref{eq. min repr of co}), it is said to originate $f^{**}$ at $x$.
\end{lemma}

\textit{Remark:} The assumption $f^{**} \left( x \right) = \inf f$ is less restrictive than it might seem since for $x' \in X'$ holds $\left( f + x' \right)^{**} = \left( f^* - x' \right)^* = f^{**} + x'$.

\begin{proof}
	Clearly, we may assume that $f$ is proper. We start by checking that the right-hand side in (\ref{eq. inf repr of co}) and hence also in (\ref{eq. min repr of co}) is not less than $f^{**} \left( x \right)$. Let $\mu_y \in \Prob \left( X \right)$ with expectation $y$. Since $f$ has a non-empty compact sublevel set by properness, it assumes a finite global minimum value. Therefore we may assume $f \ge 0$ so that $f^{**} = \overline{\ch} f$ is a convex, lower semi-continuous, proper function and the Jensen inequality yields
	\begin{equation} \label{eq. Jensen application}
		f^{**} \left( y \right)
		= f^{**} \left( \int \om \, d \mu_y \left( \om \right) \right)
		\le \int f^{**} \, d \mu_y
		\le \int f \, d \mu_y.
	\end{equation}
	Taking the infimum over such $\mu_y$ and sending $y \to x$ yields
	$$
	f^{**}( x ) \le \liminf_{y \to x} \inf \left\{ \int f \, d \mu \st \mu \in \Prob(X), \int \om \, d \mu \left( \om \right) = y \right\}.
	$$
	For proving the converse inequality we may assume
	\begin{equation} \label{eq. finite value}
		f^{**} \left( x \right) < \i.
	\end{equation}
	By \cite[Thm. 4.84 and Thm. 4.92(iii)]{CalcVar Lp} holds
	\begin{align*}
		f^{**} \left( x \right)
		& = \liminf_{y \to x} \ch f \left( y \right) \\
		& = \liminf_{y \to x} \inf\left\{ \sum_{k = 1}^{N} \lambda_k f \left( y_k \right) \st N \in \N, \sum_{k = 1}^{N} \lambda_k = 1, \lambda_k \ge 0, \sum_{k = 1}^{N} \lambda_k y_k = y \right\} \\
		& \ge \liminf_{y \to x} \inf \left\{ \int f \, d \mu \st \mu \in \Prob(X), \int \om \, d \mu \left( \om \right) = y \right\}
	\end{align*}
	so that (\ref{eq. inf repr of co}) has been proved. In the last step we used that
	$$
	\sum_{k = 1}^{N} \lambda_k f \left( y_k \right) = \int f \left( \om \right) \, d \sum_{k = 1}^{N} \lambda_k \delta_{y_k} \left( \om \right).
	$$
	Regarding the remaining claims, we need to deduce more precise information from the representation
	\begin{equation} \label{eq. Fonseca/Leoni representation}
		f^{**} \left( x \right)
		= \liminf_{y \to x} \inf \left\{ \sum_{k = 1}^{N} \lambda_k f \left( y_k \right) \st N \in \N, \sum_{k = 1}^{N} \lambda_k = 1, \lambda_k \ge 0, \sum_{k = 1}^{N} \lambda_k y_k = y \right\}.
	\end{equation}
	Let $\mathcal{N} \left( x \right)$ be the neighbourhood filter of $x$. Remember that in a general topological space, the lower limit is defined via
	$$
	\liminf_{y \to x} \ch f \left( y \right) = \sup_{U \in \mathcal{N} \left( x \right) } \inf_U \ch f.
	$$
	Hence, we find a sequence $U_n \in \mathcal{N} \left( x \right)$ such that $a_n = \inf_{U_n} \ch f$ has $f^{**} \left( x \right) = \sup_n a_n = \lim_n a_n$. Let $\left( V_i \right)_{i \in I} \subset \mathcal{N} \left( x \right)$ be a base of neighbourhoods directed by inclusion. For any $\alpha = \left( n , i \right) \in \N \times I =: A$, we may by (\ref{eq. Fonseca/Leoni representation}) find a convex combination $\mu_\alpha$ of Dirac measures
	$$
	\mu_\alpha = \sum_{k = 1}^{N \left( \alpha \right) } \lambda_k^\alpha \delta_{y_k^\alpha};
	\qquad
	y^\alpha = \sum_{k = 1}^{N \left( \alpha \right)} \lambda_k^\alpha y_k^\alpha \in U_n \cap V_i
	$$
	with
	\begin{equation} \label{eq. value approximation}
		\frac{1}{n} + f^{**} \left( x \right) \ge \int f \, d \mu_\alpha \ge a_n.
	\end{equation}
	This defines a net $\mu_\alpha$ of discrete probability measures with
	\begin{enumerate}[label = (\roman*)]
		\item $\lim_{\alpha \in A } y^\alpha = x$; \label{en. it. expectation converges}
		\item $f^{**} \left( x \right) = \lim_{\alpha \in A} \int f \, d \mu_\alpha$. \label{en. it. values converges}
	\end{enumerate}
	If $\mu_\alpha$ has a weakly convergent subnet $\mu_\beta \weaklyto \mu_{x}$ whose limit has expectation $x$, then we can conclude that $\mu_{x}$ originates $f^{**}$ at $x$ and hence (\ref{eq. min repr of co}) holds: From Proposition \ref{prop. lsc. integrand induces lsc. functional} and (\ref{eq. Jensen application}) we then have
	$$
	f^{**} \left( x \right) = \lim_{\beta \in J} \int f \, d \mu_{\beta} \ge \int f \, d \mu_{x} \ge f^{**} \left( x \right).
	$$
	This uses that $f$ has closed sublevel sets and hence is lower semi-continuous. Therefore it remains to show that (a) $\mu_\alpha$ admits a weakly convergent subnet (b) whose limit has expectation $x$. Compactness of the sublevel sets $S_r = \left\{ y \in X \st f \left( y \right) \le r \right\}$ together with $f \ge 0$ and (\ref{eq. finite value}) yields the uniform tightness estimate
	\begin{equation} \label{eq. uniform tightness estimate}
		r \sup_\alpha \mu_\alpha \left( S_r^c \right)
		\le \sup_\alpha \int_{S_r^c} f \, d \mu_\alpha
		\le f^{**} \left( x \right) < \i \quad \forall \, r > 0.
	\end{equation}
	By (\ref{eq. uniform tightness estimate}) we may invoke \cite[Thm. 4.5.3]{Conv Measures} to deduce existence of a convergent subnet $\mu_{\beta} \weaklyto \mu_{x}$ weakly in $\Prob(Y)$ with $\mu_{x}$ a Radon measure, hence $\tau$-additive. Regarding the expectation of $\mu_{x}$, consider first the case $f^{**} \left( x \right) = \inf f$. In this case we have from Proposition \ref{prop. lsc. integrand induces lsc. functional} and \ref{en. it. values converges} that
	$$
	\inf f = f^{**} \left( x \right) = \lim_{\beta \in J} \int f \, d \mu_\beta \ge \int f \, d \mu_{x} \ge \inf f.
	$$
	Consequently
	\begin{equation} \label{eq. limit convergence}
		\lim_{\beta \in J} \int f \, d \mu_\beta = \int f \, d \mu_{x}.
	\end{equation}
	Let $U' \subset \dom \left( f^* \right)$ be a balanced absorbent subset and $u' \in U'$. By Proposition \ref{prop. lsc. integrand induces lsc. functional} and the lower bound $f - u' \ge - f^* \left( u' \right)$ holds
	\begin{equation}\label{eq. lower semi-continuity estimate}
		\liminf_{\beta \in J} \int f - u' \, d \mu_\beta \ge \int f - u' \, d \mu_{x}.
	\end{equation}
	As Proposition \ref{prop. lsc. integrand induces lsc. functional} implies $\int f \, d \mu_{x} \le \liminf_{\beta \in J} \int f \, d \mu_\beta$, the terms $\int f \, d \mu_{x}$ and $\liminf_{\beta \in J} \int f \, d \mu_\beta$ are finite by \ref{eq. value approximation}. Moreover, $\lim_{\beta \in J} \int u' \, d \mu_\beta$ exists by \ref{en. it. expectation converges}. Therefore we  may equivalently rearrange (\ref{eq. lower semi-continuity estimate}) to obtain
	$$
	\liminf_{\beta \in J} \int f \, d \left( \mu_\beta - \mu_{x} \right) \ge \lim_{\beta \in J} \int u' \, d \left( \mu_\beta - \mu_{x} \right).
	$$
	In particular $0 \ge \lim_{\beta \in J} \int u' \, d \left( \mu_\beta - \mu_{x} \right)$ by (\ref{eq. limit convergence}) so that absorbency of $U'$ implies $u' \left( x \right) = \lim_{\beta \in J} \int u' \, d \mu_\beta = \int u' \, d \mu_{x}$ for all $u' \in X'$, i.e. $\mu_{x}$ has expectation $x$ in the sense of Pettis' integral. Finally, consider the case $\dom \left( f^* \right) = X'$. Arguing as before we obtain that
	\begin{equation} \label{eq. dominance over linear}
		\liminf_{\beta \in J} \int f \, d \left( \mu_\beta - \mu_{x} \right) \ge \limsup_{\beta \in J} \int x' \, d \left( \mu_\beta - \mu_{x} \right) \quad \forall \, x' \in X'.
	\end{equation}
	The upper bound in (\ref{eq. dominance over linear}) being finite, this is impossible unless $\lim_{\beta \in J} \int x' \, d \mu_{\beta} = \int x' \, d \mu_{x}$ for all $x' \in X'$.
\end{proof}

A very intuitive consequence of Lemma \ref{lem. biconjugate representation} is

\begin{corollary} \label{cor. relation minimizer sets}
	For $f$ as in Lemma \ref{lem. biconjugate representation} with $\dom \left( f^* \right)$ containing a balanced absorbent subset holds
	\begin{equation} \label{eq. minimizer set convex hull}
		\overline{\ch} \Argmin_{x \in X} f \left( x \right) = \Argmin_{x \in X} f^{**} \left( x \right).
	\end{equation}
\end{corollary}

\begin{proof}
	$\subseteq$: As $f^{**}$ is convex and lower semi-continuous, this follows from $\inf f = \inf f^{**}$.
	
	$\supseteq$: For $x \in \Argmin_{x \in X} f^{**} \left( x \right)$ exists $\mu_{x} \in \Prob(X)$ originating $f^{**}$ at $x$ by Lemma \ref{lem. biconjugate representation}. We have
	$$
	\int f \, d \mu_{x} = f^{**} \left( x \right) = \inf f^{**} = \inf f
	$$
	so that $\mu_{x}$ is concentrated on $\Argmin_{x \in X} f \left( x \right)$ and therefore
	$$
	x = \int \om \, d \mu_{x} \left( \om \right) \in \overline{\ch} \Argmin_{x \in X} f \left( x \right).
	$$
\end{proof}

The next lemma is not a corollary to Lemma \ref{lem. biconjugate representation}, but nevertheless adds to the utility of Lemma \ref{lem. biconjugate representation} by elucidating its consequences.

\begin{lemma} \label{lem. properties of orginating measure}
	Let $f \colon X \mapsto \left( - \i , \i \right]$ be Borel measurable and let $\mu_x \in \Prob \left( X \right)$ originate $f^{**}$ at $x$. For any sequence $\ell_n$ of affine continuous functions with $\ell_n \le f^{**}$ and $\lim_{n \to \i} \ell_n \left( x \right) = f^{**} \left( x \right)$, the measure $\mu_{x}$ is concentrated on the set
	$$
	A = \left\{ a \in X \st f \left( a \right) = f^{**} \left( a \right) \text{ and } \lim_n \ell_n \left( a \right) = f^{**} \left( a \right) \right\}
	$$
	and $f^{**}$ is affine on $\ch A$. Moreover
	$$
	f^{**} \left( a \right) = f^{**} \left( x \right) + \langle x' , a - x \rangle \quad \forall \, a \in \overline{\ch} A
	$$
	if $\ell_n$ is a constant sequence $\ell \left( x \right) = f^{**} \left( x \right) + \langle x' , x - x \rangle$ with $x' \in \p f^{**} \left( x \right)$.
\end{lemma}

\begin{proof}
	For any choice of $\ell_n$ we have $\mu_{x} \left( A \right) = 1$ since
	\begin{align*}
		\int f^{**} \, d \mu_{x}
		\le \int f \, d \mu_{x}
		= f^{**} \left( x \right)
		= \lim_n \ell_n \left( x \right)
		& = \lim_n \int \ell_n \, d \mu_{x} \\
		& \le \int f^{**} \, d \mu_{x}.
	\end{align*}
	Affinity follows by convexity of $f^{**}$ and since for $a_0, a_1 \in A$ and $\lambda \in \left( 0 , 1 \right)$ holds
	\begin{align*}
		\lambda f^{**} \left( a_1 \right) + (1 - \lambda) f^{**} \left( a_0 \right)
		& = \lim_n \ell_n \left( \lambda a_1 + \left( 1 - \lambda \right) a_0 \right) \\
		& \le f^{**} \left( \lambda a_1 + \left( 1 - \lambda \right) a_0 \right).
	\end{align*}
	The last claim follows by taking $\ell_n$ as the constant sequence $\ell \colon a \mapsto f^{**} \left( x \right) + \langle x' , a - x \rangle$ for $x' \in \p f^{**} \left( x \right)$ and using that $\left\{ f^{**} = \ell \right\} = \left\{ f^{**} \le \ell \right\}$ is closed by lower semi-continuity.
\end{proof}

\section{The Main Theorem}

The stage has been set for

\begin{theorem} \label{thm. tilted uniqueness iff essential strict convexity}
	Let $X$ be a locally convex Hausdorff space and $f \colon X \mapsto \left( - \i , \i \right]$ a function such that, for each $x' \in \dom \left( \p f^* \right)$, the tilted function $f - x'$ has compact sublevel sets and $\dom \left( f^* \right) - x'$ contains a balanced absorbent subset. The following are equivalent:
	\begin{enumerate}[label = (\roman*)]
		\item For all $x' \in X'$ exists at most one $\bar{x} \in \Argmin_{x \in X} f(x) - \langle x' , x \rangle_X$. \label{enum. it. 1 thm. tilted uniqueness iff essential strict convexity}
		\item $f^{**}$ is essentially strictly convex and agrees with $f$ on $\dom \left( \p f^{**} \right)$. \label{enum. it. 2 thm. tilted uniqueness iff essential strict convexity}
	\end{enumerate}
\end{theorem}

\textit{Remark:} If $X$ besides its locally convex topology $\sigma$ carries a Banach space topology $\tau$ such that $X'_\sigma = X'_\tau$, then \ref{enum. it. 2 thm. tilted uniqueness iff essential strict convexity} implies that $f$ and $f^{**}$ agree globally. This follows from the Br{\o}ndsted-Rockafellar Theorem \cite[Thm. 2]{Subgradients}: If $X'_\sigma = X'_\tau$, then the $\sigma$-subgradients and $\tau$-subgradients of $f^{**}$ coincide. By \cite[Thm. 2]{Subgradients} one may reconstruct a convex, lower semi-continuous, proper function $f$ on a Banach space as the lower semi-continuous envelope of the function
$$
\tilde{f} \left( x \right) =
\begin{cases}
	f \left( x \right) & \text{ if } x \in \dom \left( \p f \right); \\
	\i & \text{ else};
\end{cases}
$$
so that then $f \ge f^{**}$, $f^{**} \left( x \right) = f \left( x \right)$ for all $x \in \dom \left( \p f^{**} \right)$ and lower semi-continuity of $f$ together imply $f = f^{**}$.

\begin{proof}
	$\ref{enum. it. 1 thm. tilted uniqueness iff essential strict convexity} \implies \ref{enum. it. 2 thm. tilted uniqueness iff essential strict convexity}$: A function is essentially strictly convex iff it is convex everywhere and is not affine on any line segment where it is subdifferentiable. We argue by contradiction: Let
	$$
	\left[ x, y \right] \subset \dom \left( \p f^{**} \right) \text{ with } z = \frac{x + y}{2}
	$$
	and suppose $f^{**}$ were affine on $\left[ x, y \right]$ and pick $z' \in \p f^{**} \left( \frac{x + y}{2} \right)$. Definition of the subdifferential and affinity yield
	$$
	f^{**}(y) \ge f^{**}(z) + \langle z' , y - z \rangle = \frac{1}{2} \left[ f^{**} \left( x \right) + f^{**} \left( y \right) + \langle z' , y - x \rangle \right].
	$$
	Consequently
	$$
	f^{**} \left( y \right) \ge f^{**} \left( x \right) + \langle z' , y - x \rangle
	$$
	so that
	\begin{align*}
		f^{**} \left( a \right)
		\ge f^{**} \left( z \right) + \langle z' , a - z \rangle
		& \ge \frac{1}{2} \left[ f^{**} \left( x \right) + f^{**} \left( y \right) \right] + \langle z' , a - z \rangle \\
		& \ge f^{**} \left( x \right) + \langle z' , a - x \rangle \quad \forall \, a \in X.
	\end{align*}
	In total $z' \in \p f^{**} \left( x \right)$. As $f^{**} - z'$ has at most one minimizer by \ref{enum. it. 1 thm. tilted uniqueness iff essential strict convexity} and Corollary \ref{cor. relation minimizer sets}, we find $\left[ x , y \right] = \{ z \}$ whence essential strict convexity follows.
	
	Second, we prove that $f$ and $f^{**}$ agree on $\dom \left( \p f^{**} \right)$. Let $x \in \dom \left( \p f^{**} \right)$. By the Fenchel-Young identity, one has $x' \in \dom \left( \p f^* \right)$ for all $x' \in \p f^{**} \left( x \right)$. Therefore our assumption implies that $\dom \left( f^* \right) - x'$ contains a balanced absorbent set. Applying Lemma \ref{lem. biconjugate representation} to the function $f - x'$, we find $\mu_x \in \Prob \left( X \right)$ originating $f^{**}$ at $x$. By Lemma \ref{lem. properties of orginating measure} exists $A$ such that $\mu_x \left( A \right) = 1$ and
	$$
	f^{**} \left( y \right) = f^{**} \left( x \right) + \langle x' , y - x \rangle \quad \forall \, y \in \overline{\ch} A.
	$$
	Therefore
	$$
	x' \in \p f^{**} \left( y \right) \quad \forall \, y \in \overline{\ch} A.
	$$
	As $f^{**} - x'$ has at most one minimizer by Corollary \ref{cor. relation minimizer sets}, essential strict convexity of $f^{**}$ implies that $\overline{\ch} A = \{ x \}$ whence $f^{**} \left( x \right) = f \left( x \right)$ for all $x \in \dom \left( \p f^{**} \right)$ follows.
	
	$\ref{enum. it. 2 thm. tilted uniqueness iff essential strict convexity} \implies \ref{enum. it. 1 thm. tilted uniqueness iff essential strict convexity}$: Let $\bar{x}_0, \bar{x}_1 \in \Argmin_{x \in X} f - x'$. From $\left( f - x' \right)^{**} = f^{**} - x'$ follows $\bar{x}_0, \bar{x}_1 \in \Argmin_{x \in X} f^{**} - x'$ by Corollary \ref{cor. relation minimizer sets}. Hence $\bar{x}_0 = \bar{x}_1$ by essential strict convexity of $f^{**}$.
\end{proof}

We conclude our investigations by keeping the promise of demonstrating how \cite[Thm. 1]{Adequate Functions} follows from Theorem \ref{thm. tilted uniqueness iff essential strict convexity}. We star with a

\begin{proposition} \label{prop. in reflexive B-space}
	Let $X$ be a reflexive Banach space, $J \colon X \mapsto \left( - \i , \i \right]$ a weakly lower semi-continuous function and $M J \colon X' \rightrightarrows X \colon x' \mapsto \Argmin_{x \in X} J \left( x \right) - \langle x', x \rangle$. If $J$ is essentially strictly convex, then
	\begin{equation} \label{eq. fake adequacy}
		\begin{cases}
			& \dom M J = \dom \left( \p J^* \right); \\
			& MJ \text{ is single-valued on its domain}.
		\end{cases}
	\end{equation}
	If $\dom \left( \p J^* \right) = \interior \dom \left( \p J^* \right) \not = \emptyset$, the converse is true as well.
\end{proposition}

\begin{proof}
	$\implies$: If $J$ is essentially strictly convex, then $MJ$ is single valued as explained in the introduction. Convexity of $J$ implies $x \in MJ \left( x' \right) \iff x' \in \p J \left( x \right) \iff x \in \p J^* \left( x' \right)$ so that $MJ = \p J^*$.
	
	$\impliedby$: Let $\overline{B}_\e \left( x' \right) \subset \dom \left( \p J^* \right)$. Then, as in the proof of \cite[Prop. 1]{Adequate Functions}, we see that since $J^*$ must be continuous at $x'$ that there exist $r, \alpha > 0$ such that
	$$
	J^* \le I_{\overline{B}_r \left( x' \right) } + \alpha.
	$$
	Taking convex conjugates of this inequality, we get
	$$
	J \ge J^{**} \ge x'_0 + r \| \cdot \|_X - \alpha.
	$$
	Consequently, the function $x \mapsto G(x) = J(x) - \langle x' , x \rangle$ is coercive, i.e. its sublevel sets are (weakly) compact. As $G^* \left( y' \right) = J^* \left( y' - x' \right)$ we may apply Theorem \ref{thm. tilted uniqueness iff essential strict convexity} to conclude that $J = \overline{\ch} J$ is essentially strictly convex.
\end{proof}

In \cite{Adequate Functions}, the function $J$ is said to be adequate if, in addition to (\ref{eq. fake adequacy}), the set $\dom \left( \p J^* \right)$ is non-empty and open. Moreover, $J$ is essentially strictly convex in the sense of \cite{Essential in B-Spaces, Adequate Functions} if in addition to $J$ being essentially strictly convex in the sense of \cite{Convex Analysis} the function $MJ$ is locally bounded on its domain. We can now obtain \cite[Thm. 1]{Adequate Functions} as a particular case of Theorem \ref{thm. tilted uniqueness iff essential strict convexity}:

\begin{theorem}
	Under the assumptions of Proposition \ref{prop. in reflexive B-space}, $J$ is adequate in the sense of \cite{Essential in B-Spaces, Adequate Functions} iff $J$ is essentially strictly convex in the sense of \cite{Essential in B-Spaces, Adequate Functions}.
\end{theorem}

\begin{proof}
	$\implies$: Proposition \ref{prop. in reflexive B-space} implies that $J$ is essentially strictly convex. To prove that $J$ also is essentially strictly convex in the sense of \cite{Essential in B-Spaces, Adequate Functions}, it suffices to observe that $MJ = \p J^*$ by the Fenchel-Young identity. Now the statement follows as the maximal monotone operator $\p J^*$ is locally bounded on its open domain.
	
	$\impliedby$: $J$ being essentially strictly convex, Proposition \ref{prop. in reflexive B-space} yields (\ref{eq. fake adequacy}). Moreover, we have $MJ = \p J^*$. Since $\dom \left( J^* \right)$ is convex, the closure of $\dom \left( \p J^* \right)$ is convex by the the Br{\o}ndsted-Rockafellar Theorem \cite[Thm. 2]{Subgradients} so that by \cite[Remarks on Ch. 2]{Convex Differentiability} each point where $\p J^*$ is locally bounded belongs to $\interior \dom \left( \p J^* \right)$. Since we assume $\p J^*$ to be locally bounded, it follows that $\dom \left( \p J^* \right)$ is open. Since $J$ is proper, so is $J^*$ and hence $\dom \left( \p J^* \right)$ is non-empty.
\end{proof}

\end{document}